\documentclass[a4paper,12pt]{amsart}
\usepackage{amsmath,amssymb,amsthm}
\usepackage{amscd}
\usepackage{array,enumerate}
\newtheorem{Thm}{Theorem}[section]
\newtheorem{Prop}[Thm]{Proposition}
\newtheorem{Lem}[Thm]{Lemma}

\newtheorem{Thmint}{Theorem}[section]
\newtheorem{Corint}[Thmint]{Corollary}

\theoremstyle{definition}
\newtheorem{Rem}[Thm]{Remark}

\newtheorem{Exm}[Thm]{Example}

\newcommand{\Cs}{\mbox{${\rm C}^\ast$}}
\newcommand{\id}{\mbox{\rm id}}

\title[Decreasing intersection of $\mathcal{O}_2$ and hyperbolic groups]{Realization of hyperbolic group \Cs-algebras as decreasing intersection of Cuntz algebras $\mathcal{O}_2$}
\author{Yuhei Suzuki}
\address{Department of Mathematical Sciences,
University of Tokyo, Komaba, Tokyo, 153-8914, Japan}
\address{Research Institute for Mathematical Sciences, Kyoto University, Kyoto, 606-8502, Japan}
\email{suzukiyu@ms.u-tokyo.ac.jp}
\begin{document}
\begin{abstract}
It is proved that for every ICC group which is embeddable into a hyperbolic group,
the reduced group ${\rm C}^\ast$-algebra is realized as
the intersection of a decreasing sequence of isomorphs of the Cuntz algebra $\mathcal{O}_2$.
The proof is based on the study of amenable quotients of the boundary actions.
\end{abstract} 
\maketitle

\section{Introduction}
It is well-known that every exact group admits an amenable action on a compact space \cite{Oz2},
and each such action gives rise to an ambient nuclear \Cs -algebra of the reduced group \Cs -algebra
via the crossed product construction \cite{Ana0}.
More generally, it is known that every separable exact \Cs -algebra is embeddable into the Cuntz algebra $\mathcal{O}_2$ \cite{KP}.
Motivated by these phenomena, we are interested in the following question.
How small can we take an ambient nuclear \Cs -algebra for a given exact \Cs -algebra?
See \cite{Oz0} for a relevant topic.
In the present paper, we give an answer to the question for the reduced group \Cs -algebras of hyperbolic groups.
The main theorem of the present paper (Theorem \ref{Thmint:Main}) states that an ambient nuclear \Cs -algebra of the reduced hyperbolic group \Cs -algebras
can be arbitrarily small in a certain sense.
\begin{Thmint}[Theorem \ref{Thm:Main}]\label{Thmint:Main}
Let $\Gamma$ be an ICC group which is
embeddable into a hyperbolic group.
Then the reduced group ${\rm C}^\ast$-algebra ${\rm C}_{\rm r}^\ast(\Gamma)$
is realized as the intersection of a decreasing sequence of isomorphs of the Cuntz algebra $\mathcal{O}_2$.
\end{Thmint}
Note that since the decreasing intersection of injective von Neumann algebras
is injective,
the analogous results for von Neumann algebras can never be true.

As immediate consequences of Theorem \ref{Thmint:Main}, we obtain the following results.
\begin{Corint}\label{Corint:decreasing&LP}
The following statements are true.
\begin{enumerate}[\upshape (1)]
\item Nuclearity of \Cs -algebras is not preserved by taking the intersection of a decreasing sequence.
\item The decreasing intersection of \Cs -algebras with the lifting property does
not need to have the local lifting property.
\end{enumerate}
\end{Corint}

The proof is based on the study of amenable actions and
approximation theory of discrete groups and \Cs -algebras.
We refer the reader to \cite{BO} for both topics.
To prove Theorem \ref{Thmint:Main},
we first construct a sequence of amenable quotients of the boundary action which tends
to the trivial action on a singleton.
Then using the AP (see \cite[Chapter 12.4]{BO}), we show that their crossed products form a decreasing sequence of
simple unital separable nuclear \Cs -algebras whose intersection is the reduced group \Cs -algebra.
(In fact they can be taken to be Kirchberg algebras. See Remark \ref{Rem:Kirchberg}.)
Finally, to make the terms of the sequence to be isomorphic to the Cuntz algebra $\mathcal{O}_2$, we apply
a deep theorem of Kirchberg and Phillips \cite{KP}.

\subsection*{Structure of the paper}
We first deal with the free groups (Section \ref{Sec:free}).
This is because the proof for the general case is technically more involved,
but the essential idea already appears in that of the free group case.
Then we complete the proof for the general case (Section \ref{Sec:gen}).
In Appendix A, we give computations of the $K$-theory of some amenable quotients
of the boundary actions of the free groups.

\subsection*{Notation}
The symbol `$\otimes$' stands for the minimal tensor product.
The symbols `$\rtimes$' and `$\bar{\rtimes}$' stand for the reduced crossed product of \Cs -algebras
and that of von Neumann algebras respectively.
For a discrete group $\Gamma$ and $g\in \Gamma$,
$\lambda_g$ denotes the unitary element of the reduced group \Cs -algebra ${\rm C}_{\rm r}^\ast(\Gamma)$ corresponding to $g$ .
For a unital $\Gamma$-\Cs-algebra $A$ and $g\in \Gamma$,
denote by $u_g$ the canonical implementing unitary element of $g$ in
the reduced crossed product $A\rtimes \Gamma$.
With the same setting, for $x\in A\rtimes \Gamma$ and $g\in \Gamma$,
the $g$th coefficient $E(xu_g^{\ast})$ of $x$ is denoted by
$E_g(x)$. Here $E\colon A\rtimes\Gamma\rightarrow A$ denotes the canonical conditional expectation.
For a set $X$, denote by $\Delta_X$
the diagonal set $\{(x, x):x\in X\}$ of $X\times X$.
For a subset $Y$ of a topological space $X$,
denote by ${\rm int}(Y)$ and ${\rm cl}(Y)$
the interior and the closure of $Y$ in $X$ respectively.
\section{Free group case}\label{Sec:free}
In this section, we fix a natural number $d$ greater than $1$.
First we review the boundary
of the free group.
Let $S$ be the set of all canonical generators of $\mathbb{F}_d$.
Define a map $\mathbb{F}_d \rightarrow \prod_{n=1}^\infty \left(S\sqcup S^{-1}\sqcup \{e\}\right)$
by $x=s_{i_1}\cdots s_{i_n}\mapsto (s_{i_1}, \ldots, s_{i_n}, e, e, \ldots)$.
Here the product $s_{i_1}\cdots s_{i_n}$ denotes the reduced form of $x$.
Then the closure of its image defines the compactification $\overline{\mathbb{F}_d}$ of $\mathbb{F}_d$.
We define the boundary $\partial \mathbb{F}_d$ of $\mathbb{F}_d$
to be the (closed) subspace $\overline{\mathbb{F}_d}\setminus\mathbb{F}_d$ of $\overline{\mathbb{F}_d}$.
Note that the boundary $\partial \mathbb{F}_d$
is naturally identified with the space of all one-sided infinite reduced words of $S$.
The left multiplication action of $\mathbb{F}_d$ on itself
extends to an action of $\mathbb{F}_d$ on $\overline{\mathbb{F}_d}$ and
its restriction defines
the action of $\mathbb{F}_d$ on $\partial \mathbb{F}_d$.
This is called the boundary action of $\mathbb{F}_d$.
It is well-known that the boundary action is (topologically) amenable \cite[Chapter 5.1]{BO}.
For any $w \in \mathbb{F}_d\setminus \{e\}$,
it is not difficult to check that the sequence $(w^n)_{n=1}^\infty$
converges to an element of the boundary.
We denote the limit by $w^{+\infty}$.
Note that $w^{+\infty}\neq w^{-\infty}$
and the set $\{w^{\pm \infty}\}$ coincides with the
set of all fixed points of $w$ in $\partial \mathbb{F}_d$.
In particular, this shows that the boundary action of $\mathbb{F}_d$ is topologically free.
Here recall that a topological dynamical system $\alpha\colon \Gamma\curvearrowright X$
is said to be topologically free if for any $g\in \Gamma\setminus\{e\}$,
the set $\{x\in X: \alpha(g)(x)\neq x\}$ is dense in $X$.

The following well-known proposition gives a criteria for Hausdorffness of a quotient space.
For completeness of the paper, we include a proof.
\begin{Prop}
Let $X$ be a compact Hausdorff space.
Let $\mathcal{R}$ be an equivalence relation on $X$.
Assume  that each equivalence class of $\mathcal{R}$ is compact and
the quotient map $\pi \colon X\rightarrow X/\mathcal{R}$ is closed.
Then the quotient space $X/\mathcal{R}$ is Hausdorff.
\end{Prop}
\begin{proof}
Let $x, y\in X/\mathcal{R}$ be two distinct elements.
Then $\pi^{-1}(\{x \})$ and $\pi^{-1}(\{ y\})$ are disjoint and both of them are compact by assumption.
Hence there are open subsets $U$ and $V$ of $X$ satisfying $\pi^{-1}(\{ x\}) \subset U$, $\pi^{-1}(\{ y \})\subset V$,
and $U\cap V=\emptyset$.
Since $\pi$ is a closed map,
both $\pi(X\setminus U)$ and $\pi(X\setminus V)$ are closed in $X/\mathcal{R}$.
Define $U_0$ and $V_0$ to be the complements of these two sets respectively.
Then both are open, $x\in U_0$, $y\in V_0$,
$\pi^{-1}(U_0)\subset U$, $\pi^{-1}(V_0)\subset V$.
In particular, they are disjoint.
\end{proof}
\begin{Lem}\label{Lem:closedness of R}
Let $W \subset \mathbb{F}_d\setminus \{e\}$ be a finite subset.
Then the set
$$\mathcal{R}_W:= \Delta_{\partial \mathbb{F}_d} \cup \left\{(gw^{+\infty}, gw^{-\infty}) : g\in \mathbb{F}_d, w\in W\cup W^{-1} \right\}$$
is an $\mathbb{F}_d$-invariant equivalence relation on $\partial\mathbb{F}_d$.
Moreover, the quotient space $\partial \mathbb{F}_d/\mathcal{R}_W$ is a Hausdorff space.
\end{Lem}
\begin{proof}
Clearly $\mathcal{R}_W$ is $\mathbb{F}_d$-invariant.
To show that the $\mathcal{R}_W$ is an equivalence relation,
it suffices to show that for $g, h\in \mathbb{F}_d$ and $w, z\in W$,
two sets $\{gw^{\pm\infty}\}$ and $\{hz^{\pm\infty}\}$
are disjoint or the same.
Assume that the two sets $\{gw^{\pm\infty}\}$ and $\{hz^{\pm\infty}\}$ intersect.
This means that the two elements $gwg^{-1}$ and $hzh^{-1}$ have a common fixed point.
Then by amenability of the boundary action,
the set $\{ gwg^{-1}, hzh^{-1}\}$ generates an amenable subgroup,
which must be isomorphic to $\mathbb{Z}$.
This shows $\{gw^{\pm\infty}\}=\{hz^{\pm\infty}\}$.
Note that this implies that each equivalence class of $\mathcal{R}_W$ contains
at most two elements.

To show Hausdorffness of the quotient space $\partial \mathbb{F}_d/\mathcal{R}_W$, since 
each equivalence class of $\mathcal{R}_W$ is finite, it suffices to show that
the quotient map $\pi\colon \partial \mathbb{F}_d\rightarrow \partial \mathbb{F}_d/\mathcal{R}_W$
is closed.
Let $A$ be a closed subset of $\partial \mathbb{F}_d$.
Then $\pi^{-1}(\pi(A))=A\cup B$, where
$$B:=\left\{gw^{-\infty}\in \partial \mathbb{F}_d: g\in \mathbb{F}_d, w\in W\cup W^{-1}, gw^{+\infty}\in A \right\}.$$
To show closedness of $\pi(A)$, which is equivalent to that of $\pi^{-1}(\pi(A))$ by definition,
it suffices to show the inclusion ${\rm cl}(B)\subset A\cup B$.
Let $x\in {\rm cl}(B)$ be given.
Take a sequence $(g_n w_n^{-\infty})_{n=1}^\infty$ of elements of $B$
which converges to $x$.
By passing to a subsequence, we may assume that
there is $w\in W\cup W^{-1}$ with $w_n=w$ for all $n\in \mathbb{N}$.
Replace $g_n$ by $g_n w^{k(n)}$ for some $k(n) \in \mathbb{Z}$
for each $n \in \mathbb{N}$, we may assume $|g_n|\leq |g_n w^k|$ for all $n \in \mathbb{N}$ and $k \in \mathbb{Z}$.
If the sequence $(g_n)_{n=1}^\infty$ has a bounded subsequence,
then it has a constant subsequence, hence $x\in B$.
Assume $|g_n|\rightarrow \infty$.
Then for each $n\in \mathbb{N}$,
by assumption on $g_n$ and $w$,
the first $(|g_n|-|w|)$th segments of $g_n w^{+\infty}$ and $g_n w^{-\infty}$ coincide with
that of $g_n$.
This shows that the limits of $(g_n w^{+\infty})_{n=1}^\infty$ and $(g_n w^{-\infty})_{n=1}^\infty$  are the same. Since $A$ is closed, we have $x\in A$ as desired.
\end{proof}
The next lemma ensures amenability of certain quotients.
\begin{Lem}\label{Lem:amenable criteria}
Let $\Gamma$ be a group,
$X$ be an amenable compact $\Gamma$-space.
Let $\mathcal{R}$ be a $\Gamma$-invariant equivalence relation of $X$
such that the quotient space $X/\mathcal{R}$ is Hausdorff.
Assume that each equivalence class of $\mathcal{R}$ is finite
and it is trivial for all but countably many classes.
Then the quotient space $X/\mathcal{R}$ is again an amenable compact $\Gamma$-space.
\end{Lem}
\begin{proof}
Recall that a compact $\Gamma$-space $X$ is amenable if and only if
the von Neumann algebra
$C(X)^{\ast \ast}\bar{\rtimes} \Gamma$ is injective \cite[Theorem 4.5]{Ana0}.
By assumption on $\mathcal{R}$, the diffuse parts of
$C(X)^{\ast \ast}$ and $C(X/\mathcal{R})^{\ast \ast}$ naturally coincide.
Therefore we only have to check injectivity of
the $\Gamma$-crossed product of the atomic part of $C(X/\mathcal{R})^{\ast \ast}$.
Note that we have
$C(X/\mathcal{R})^{\ast \ast}_{\rm atom}= \ell^{\infty} (X/\mathcal{R})$.
This is identified with the $\Gamma$-von Neumann subalgebra of
$C(X)^{\ast \ast}_{\rm atom}= \ell^{\infty}(X)$
consisting of functions which are constant on each equivalence class of $\mathcal{R}$.
Since each equivalence class is finite,
the averaging on each equivalence class gives
a $\Gamma$-equivariant normal conditional expectation
from $C(X)^{\ast\ast}_{\rm atom}$ onto $C(X/\mathcal{R})^{\ast \ast}_{\rm atom}$.
This shows injectivity of $C(X/\mathcal{R})^{\ast \ast}_{\rm atom} \bar{\rtimes} \Gamma$.
\end{proof}

The next lemma is well-known.
For the reader's convenience, we include a proof.
\begin{Lem}\label{Lem:density}
The equivalence relation $\mathcal{R}:=\bigcup_{w \in \mathbb{F}_d \setminus\{e\}} \mathcal{R}_{\{w\}}$
is dense in $\partial \mathbb{F}_d \times \partial \mathbb{F}_d$.
\end{Lem}
\begin{proof}
For any two distinct elements $x, y\in \mathbb{F}_d$ with $|x|=|y|$,
take $w\in \mathbb{F}_d$ satisfying $|xwy^{-1}|=|x|+|w|+|y|$.
Then for any $n\in \mathbb{N}$,
the first $|x|$th segments of $(xwy^{-1})^n$, $(xwy^{-1})^{-n}$ are equal to
$x$, $y$, respectively.
Therefore the same hold true for $(xwy^{-1})^{+\infty}$, $(xwy^{-1})^{-\infty}$,
respectively.
This proves density of $\mathcal{R}$.
\end{proof}
The essential idea of the proof of the next proposition already appears in \cite{Zac}.
\begin{Prop}\label{Prop:AP}
Let $\Gamma$ be a discrete group with the AP.
Let $B \subset A$ be an inclusion of $\Gamma$-${\rm C}^\ast$-algebras.
Assume that an element $x\in A\rtimes \Gamma$ satisfies
$E_g(x) \in B$ for all $g\in \Gamma$.
Then $x$ is contained in $B\rtimes \Gamma$.
Here we identify $B\rtimes \Gamma$ with the ${\rm C}^\ast$-subalgebra
of $A\rtimes \Gamma$ in the canonical way.
\end{Prop}
\begin{proof}
Since $\Gamma$ has the AP, the proof of Theorem 12.4.9 in \cite{BO} shows that there is a net
$(\varphi_i)_{i \in I}$ of finitely supported functions on $\Gamma$ such that
the bounded linear maps
$\Phi_i\colon \mathrm{C}_{\rm r}^\ast(\Gamma)\rightarrow \mathrm{C}_{\rm r}^\ast(\Gamma)$
given by $\lambda_g\mapsto \varphi_i(g) \lambda_g$
satisfy the following property.
For every \Cs -algebra $C$, the tensor products $\Phi_i\otimes \id_C$ converge to the identity map
in the pointwise norm topology.
For $i \in I$, define the linear map
$\Psi_i \colon A\rtimes\Gamma\rightarrow A\rtimes \Gamma$
by $\Psi_i (y):=\sum_{g\in \Gamma} E_g(y)\varphi_i(g)u_g$.
We claim that the net $(\Psi_i)_{i \in I}$ converges to the identity map
in the pointwise norm topology.
To show this, consider the embedding
$\iota \colon A\rtimes\Gamma \rightarrow (A\rtimes\Gamma)\otimes {\rm C}_{\rm r}^\ast(\Gamma)$
induced from the maps $a\in A\mapsto a\otimes 1$ and
$u_g\in \Gamma \mapsto u_g \otimes \lambda_g$.
Then the composite $\iota \circ \Psi_i$
coincides with the composite $({\rm id}\otimes \Phi_i)\circ \iota$.
This proves the convergence condition.
Now let $x$ be as stated.
Then for any $i\in I$,
we have $\Psi_i(x) \in B\rtimes\Gamma$.
Since the net $(\Psi_i(x))_{i \in I}$ converges in norm to $x$,
we have $x \in B\rtimes \Gamma$ as desired.
\end{proof}
\begin{Thm}\label{Thm:intersection}
The reduced group ${\rm C}^\ast$-algebra ${\rm C}_{\rm r}^\ast(\mathbb{F}_d)$
is realized as the intersection of a decreasing sequence of simple unital separable nuclear \Cs -algebras.
\end{Thm}
\begin{proof}
Take an increasing sequence $(W_n)_{n=1}^\infty$ of finite subsets of $\mathbb{F}_d\setminus\{e\}$
which exhausts $\mathbb{F}_d\setminus\{e\}$.
Consider the ${\rm C}^\ast$-subalgebra $A_n:=C(\partial\mathbb{F}_d/ \mathcal{R}_{W_n})\rtimes \mathbb{F}_d$ of
$C(\partial \mathbb{F}_d)\rtimes \mathbb{F}_d$ for each $n\in \mathbb{N}$.
Then, for each $n\in \mathbb{N}$, by Lemmas \ref{Lem:closedness of R} and \ref{Lem:amenable criteria}, the action $\mathbb{F}_d\curvearrowright \partial \mathbb{F}_d/\mathcal{R}_{W_n}$ is
amenable. 
Its minimality and topological freeness are clear.
Therefore, by \cite[Theorem 4.5]{Ana0} and \cite{AS}, each $A_n$ is a simple unital separable nuclear \Cs -algebra.
For all $g\in \mathbb{F}_d$, we have 
$$E_g\left( \bigcap_{n=1}^\infty A_n \right)\subset \bigcap_{n=1}^\infty E_g(A_n) = \bigcap_{n=1}^\infty C(\partial\mathbb{F}_d/ \mathcal{R}_{W_n})=\mathbb{C}.$$
Here the last equality follows from Lemma \ref{Lem:density}.
Since the free group has AP \cite[Corollary 12.3.5]{BO}, Proposition \ref{Prop:AP} proves the equality
$\bigcap_{n=1}^\infty A_n={\rm C}_{\rm r}^\ast(\mathbb{F}_d)$.
\end{proof}
Corollary \ref{Corint:decreasing&LP} now follows from Theorem \ref{Thm:intersection}.
(For statement (2), recall that every separable nuclear \Cs -algebra has the lifting property \cite[Theorem 3.10]{CE},
whereas the reduced group \Cs -algebra ${\rm C}^\ast_{\rm r}(\mathbb{F}_d)$ does not have the local lifting property \cite[Corollary 3.7.12]{BO}.)
\begin{Rem}\label{Rem:Kirchberg}
In fact, we can show that the crossed product
$C(\partial \mathbb{F}_d/\mathcal{R}_W)\rtimes\Gamma$
is a Kirchberg algebra in the UCT class for any finite $W$.
This is done by showing that the action
$\Gamma\curvearrowright \partial \mathbb{F}_d/\mathcal{R}_W$ is a locally boundary action in the sense of \cite{LaS}.
\end{Rem}
\begin{Rem}\label{Rem:K-theory}
In Appendix A, we show that the $K$-theory $(K_0, [1]_0, K_1)$ of
$C(\partial\mathbb{F}_d/ \mathcal{R}_S)\rtimes \mathbb{F}_d$
is isomorphic to $(\mathbb{Z}^d, (1, 1, \ldots, 1), \mathbb{Z}^d)$.
Since the unit $[1]_0\in K_0(C(\partial\mathbb{F}_d)\rtimes \mathbb{F}_d)$ is of finite order,
even from the $K$-theoretical view point, the ambient nuclear \Cs-algebra
$C(\partial\mathbb{F}_d/ \mathcal{R}_S)\rtimes \mathbb{F}_d$ of ${\rm C}^\ast_{\rm r}(\mathbb{F}_d)$
is more tight than the boundary algebra.
\end{Rem}
Now we obtain the main theorem for free groups.
\begin{Thm}\label{Thm:Cuntz}
The reduced group ${\rm C}^\ast$-algebra ${\rm C}_{\rm r}^\ast(\mathbb{F}_d)$
is realized as the intersection of a decreasing sequence of isomorphs of the Cuntz algebra $\mathcal{O}_2$.
\end{Thm}
\begin{proof}
First take a decreasing sequence of simple unital separable nuclear \Cs -algebras $(A_n)_{n=1}^\infty$ whose intersection is isomorphic to
the reduced group \Cs -algebra ${\rm C}_{\rm r}^\ast(\mathbb{F}_d)$.
For $n\in \mathbb{N}$,
set $B_n:=\left( \bigotimes_{k=n}^\infty \mathcal{O}_2\right)\otimes A_n$.
Then, thanks to the theorem of Kirchberg and Phillips \cite[Theorem 3.8]{KP}, each $B_n$ is isomorphic to the Cuntz algebra $\mathcal{O}_2$.
Therefore it suffices to show the equality $\bigcap_{n=1}^\infty B_n =\mathbb{C}\otimes \left( \bigcap_{n=1}^\infty A_n \right)$.
To show this, for each $n\in \mathbb{N}$, take a state $\varphi_n$ on $\bigotimes_{k=n+1}^\infty \mathcal{O}_2$
and define the conditional expectation
$E_n\colon \bigotimes_{k=1}^\infty\mathcal{O}_2\rightarrow \bigotimes_{k=1}^n \mathcal{O}_2$
by $x \otimes y \mapsto \varphi_n(y)x$ for $x\in  \bigotimes_{k=1}^n \mathcal{O}_2$ and $y \in \bigotimes_{k=n+1}^\infty \mathcal{O}_2$.
Then the maps $E_n$ converge to the identity map
in the pointwise norm topology and they satisfy
$E_n(B_m)\subset \mathbb{C}\otimes A_m$ for all $n, m\in \mathbb{N}$ with $n\leq m$.
These two facts prove the inclusion
$\bigcap_{n=1}^\infty B_n \subset
\mathbb{C}\otimes \left( \bigcap_{n=1}^\infty A_n \right)$.
The converse inclusion is trivial.
\end{proof}

\section{General case}\label{Sec:gen}
In this section, we complete the proof of the main theorem.
For the definition and basic properties of hyperbolic groups,
we refer the reader to \cite[Section 5.3]{BO} and \cite{GH}.
Here we only recall a few facts.
(See 8.16, 8.21, 8.28, and 8.29 in \cite{GH}.)
For a torsion free element $t$ of a hyperbolic group $\Gamma$, the sequence $(t^n)_{n=1}^\infty$
is quasi-geodesic.
The boundary action of $t$ has exactly two fixed points.
They are the points represented by the quasi-geodesic paths $(t^n)_{n=1}^\infty$ and $(t^{-n})_{n=1}^\infty$.
We denote them by $t^{+\infty}$ and $t^{-\infty}$ respectively.
For any neighborhoods $U_{\pm}$ of $t^{\pm\infty}$,
there is $n\in \mathbb{N}$
such that for any $m\geq n$,
$t^m(\partial \Gamma\setminus U_-)\subset U_+$ holds.

For a metric space $(X, d)$ and its points $x, y, z\in X$,
we denote by $\langle y, z \rangle_x$ the Gromov product $(d(y, x)+ d(z, x) - d(y, z))/2$
of $y, z$ with respect to $x$.
\begin{Lem}
Let $\Gamma$ be a hyperbolic group.
Let  $T$ be a finite set of torsion free elements of $\Gamma$.
Then the set
$$\mathcal{R}_T:= \Delta_{\partial \Gamma}　\cup \left\{ (g.t^{+\infty}, g.t^{-\infty}): g　\in \Gamma, t \in T\cup T^{-1} \right\}$$
is a $\Gamma$-invariant equivalence relation on $\partial \Gamma$.
Moreover, the quotient space $\partial\Gamma/ \mathcal{R}_T$ is a Hausdorff space.
\end{Lem}
\begin{proof}
Clearly $\mathcal{R}_T$ is $\Gamma$-invariant.
Let $s, t$ be torsion free elements of $\Gamma$.
Then the two sets $\{s^{\pm\infty}\}$ and $\{t^{\pm\infty}\}$
are either disjoint or the same \cite[8.30]{GH}.
Therefore the set $\mathcal{R}_T$ is an equivalence relation.
Note that this shows that each equivalence class of $\mathcal{R}_T$ contains at most two points.

For Hausdorffness of the quotient space, since each equivalence class only consists finitely many points, it suffices to show
that the quotient map $\pi\colon \partial \Gamma\rightarrow \partial \Gamma/\mathcal{R}_T$
is closed.
Let $A$ be a closed subset of $\partial \Gamma$.
Then $\pi^{-1}(\pi(A))=A\cup B$,
where
$$B:=\left\{g.t^{-\infty}\in \partial \Gamma : g\in \Gamma, t\in T\cup T^{-1}, g.t^{+\infty}\in A\right\}.$$
To show closedness of $\pi(A)$, which is equivalent to that of $\pi^{-1}(\pi(A))$,
it suffices to show that ${\rm cl}(B)\subset A\cup B$.
Fix a finite generating set $S$ of $\Gamma$ and denote by $|\cdot |$ and $d(\cdot, \cdot)$
the length function and the left invariant metric on $\Gamma$ determined by $S$ respectively.
Take $\delta > 0$ with the property that every geodesic triangle in $(\Gamma, d)$ is $\delta$-thin \cite[Proposition 5.3.4]{BO}.
Let $x\in {\rm cl}(B)$ and take a sequence $(g_n.t_n^{-\infty})_{n=1}^\infty$ in $B$ which converges
to $x$.
By passing to a subsequence, we may assume that there is $t\in T\cup T^{-1}$
with $t_n=t$ for all $n\in \mathbb{N}$.
Replace $g_n$ by $g_n t^{l(n)}$ for some $l(n) \in \mathbb{Z}$ for each $n\in \mathbb{N}$,
we may further assume $|g_n|\leq |g_n t^k|$ for all $k\in \mathbb{Z}$ and $n\in \mathbb{N}$.
If the sequence $(g_n)_{n=1}^\infty$ has a bounded subsequence,
then it has a constant subsequence, hence we have $x\in B$.
Assume $|g_n|\rightarrow \infty$.
For each $k\in \mathbb{N}$, take a geodesic path $[e, t^k]$ from $e$ to $t^k$.
Since $t$ is torsion free, the sequence $(t^n)_{n=1}^\infty$ is quasi-geodesic.
Therefore, by \cite[Proposition 5.3.5]{BO}, there is $D>0$ such that the Hausdorff distance between
$[e, t^k]$ and $( e, t, \ldots, t^k)$ is less than $D$ for all $k\in \mathbb{N}$.
This shows ${\rm dist}(g_n^{-1}, [e, t^k])\geq |g_n|-D$.
Since a geodesic triangle with the vertices
$\{e, g_n^{-1}, t^k\}$ is $\delta$-thin, the above inequality
implies $|g_nt^k|\geq |t^k|+|g_n|-D-2\delta$.
The same inequality also holds for negative case.
From these two inequalities, we obtain
$\langle g_nt^{k}, g_nt^{-l} \rangle_e\geq |g_n|-D-2\delta$ for all $n, k, l \in \mathbb{N}$.
This inequality shows that
the limits of $(g_n.t^{+\infty})_{n=1}^\infty$ and $(g_n.t^{-\infty})_{n=1}^\infty$ coincide.
Since $A$ is closed, we have $x\in A$ as required.
\end{proof}

For a subgroup $\Lambda$ of a hyperbolic group $\Gamma$, we define
the limit set $L_\Lambda$ of $\Lambda$ to be the closure
of the set $\{t^{+\infty}\in \partial\Gamma: t\in\Lambda {\rm\ torsion\ free}\}$ in $\partial \Gamma$.
Recall that every hyperbolic group
does not contain an infinite torsion subgroup \cite[8.36]{GH}.
Therefore the limit set $L_\Lambda$ is nonempty when $\Lambda$ is infinite.
Clearly, the limit set $L_\Lambda$ is $\Lambda$-invariant
hence $\Lambda$ acts on $L_\Lambda$ in the canonical way.
Next we give two lemmas on the action on the limit set,
which are familiar to specialists.
\begin{Lem}\label{Lem:top}
Let $\Lambda$ be an ICC subgroup of a hyperbolic group $\Gamma$.
Then the action $\varphi_\Lambda$ of $\Lambda$ on its limit set $L_\Lambda$ is amenable, minimal, and topologically free.
\end{Lem}
\begin{proof}
Amenability of the boundary action shows that
the action $\varphi_\Lambda$ is amenable.
Since $\Lambda$ is ICC, it is neither finite nor virtually cyclic.
Hence $\Lambda$ contains a free group of rank $2$ \cite[Theorem 8.37]{GH}.
Hence there are two torsion free elements $s$ and $t$ of $\Lambda$
which do not have a common fixed point.
This shows minimality of $\varphi_\Lambda$.

Assume now that $\varphi_\Lambda$ is not topologically free.
Take an element $g_1 \in \Lambda\setminus \{e\}$ such that
the set $F_{g_1}:=\{x\in L_\Lambda:g_1.x=x\}$ has a nontrivial interior.
Since $L_\Lambda$ does not have an isolated point,
the order of $g_1$ must be finite. 
Assume $F_{g_1}=L_\Lambda$.
This means that the kernel of $\varphi_\Lambda$ is nontrivial.
Since it cannot contain a torsion free element,
it is a nontrivial normal torsion subgroup.
Therefore it must be finite.
This contradicts to the ICC condition.
For a subgroup $G$ of $\Lambda$, we set
$F_G:=\bigcap_{g \in G}F_g$.
Note that for a subgroup $G$ of $\Lambda$ and $g\in\Lambda$,
we have $F_{gGg^{-1}}=gF_G$.
Set $G_1:=\langle g_1\rangle.$
Then ${\rm int}(F_{G_1}) = {\rm int}(F_{g_1})\neq \emptyset$.
We will show that there is $g_2\in\Lambda$ satisfying
$$\emptyset \neq g_2({\rm int}(F_{G_1}))\cap{\rm int}(F_{G_1})\subsetneq {\rm int}(F_{G_1}).$$
Indeed, if such $g_2$ does not exist,
then the family $\{g({\rm int}(F_{G_1})): g\in\Lambda\}$ makes an open covering
of $L_\Lambda$ whose members are mutually disjoint.
(Note that if $g\in \Lambda$ satisfies $ {\rm int}(F_{G_1})\subsetneq g({\rm int}(F_{G_1}))$,
then $g^{-1}$ satisfies the required condition.)
This forces that the subgroup
$$\Lambda_0:=\left\{g\in \Lambda: g({\rm int}(F_{G_1}))={\rm int}(F_{G_1})\right\}$$
has finite index in $\Lambda$.
Since $\Lambda$ is ICC,
the subgroup $G:=\langle gG_1g^{-1}: g\in \Lambda_0\rangle$ must be infinite.
Moreover, by definition, we have ${\rm int}(F_{G})={\rm int}(F_{G_1})\neq 0$.
Hence $G$ must be an infinite torsion subgroup, a contradiction.
Take $g_2\in \Lambda$ as above and set $G_2=\langle G_1, g_2G_1{g_2}^{-1}\rangle$.
Then we have
$\emptyset\neq{\rm int}(F_{G_2})\subsetneq {\rm int}(F_{G_1})$.
This shows that $G_2$ is still finite and is larger than $G_1$.
Continuing this argument inductively, we obtain
a strictly increasing sequence $(G_n)_{n=1}^\infty$ of finite subgroups of $\Lambda$.
Hence the union $\bigcup_{n=1}^\infty G_n$ is an infinite torsion subgroup of $\Lambda$,
again a contradiction. 
\end{proof}

\begin{Lem}
For $\Lambda$ as in Lemma {\rm \ref{Lem:top}}, the equivalence relation
$$\mathcal{R}:=\left( \bigcup_{t\in \Lambda, {\rm\ torsion\ free}} \mathcal{R}_{\{t\}} \right) \cap (L_\Lambda\times L_\Lambda)$$ on $L_{\Lambda}$
is dense in $L_\Lambda\times L_\Lambda$.
\end{Lem}
\begin{proof}
Let $s$ and $t$ be two torsion free elements in $\Lambda$ which do not have a common fixed point.
For any neighborhoods $U_{\pm}$ of $s^{\pm \infty}$ and neighborhoods $V_\pm$ of $t^{\pm \infty}$
with the properties $U_+\cap V_-=\emptyset$ and $U_-\cap V_+=\emptyset$,
take a natural number $N$ satisfying
$s^{N}(\partial\Gamma\setminus U_-)\subset U_+$ and
$t^{N}(\partial\Gamma\setminus V_-)\subset V_+$.
Then, for any $m\in \mathbb{N}$, we have $(s^Nt^N)^m(\partial\Gamma\setminus V_-)\subset U_+$
and $(s^Nt^N)^{-m}(\partial\Gamma\setminus U_+)\subset V_-$.
This shows that $(s^Nt^N)^{+\infty}\in {\rm cl}(U_+)$ and $(s^Nt^N)^{-\infty}\in {\rm cl}(V_-)$.
Thus the product ${\rm cl}(U_+)\times {\rm cl}(V_-)$ intersects with $\mathcal{R}$.
This proves
density of $\mathcal{R}$.
\end{proof}

Now we obtain the main theorem.
\begin{Thm}\label{Thm:Main}
Let $\Gamma$ be an ICC group which is
embeddable into a hyperbolic group.
Then the reduced group ${\rm C}^\ast$-algebra ${\rm C}_{\rm r}^\ast(\Gamma)$
is realized as the intersection of a decreasing sequence of isomorphs of the Cuntz algebra $\mathcal{O}_2$.
\end{Thm}
\begin{proof}
Ozawa has shown that every hyperbolic group
is weakly amenable \cite{Oz}.
Thus $\Gamma$ is weakly amenable and in particular has the AP.
From this fact with the above three lemmas,
the proofs of Theorems \ref{Thm:intersection} and \ref{Thm:Cuntz}
now work for our case.
\end{proof}
\begin{Rem}\label{Rem:nonICC}
Even for non-ICC case,
our argument gives a realization of the reduced group \Cs -algebra
as the intersection of a decreasing sequence of nuclear \Cs -algebras.
However, this does not give a decreasing sequence of simple ones,
since the action on the limit set is not faithful when the group is not ICC.
\end{Rem}
We remark that although the reduced group \Cs -algebras in Theorem \ref{Thm:Main}
are simple due to the Powers property \cite[Proposition 11]{Ha}, the decreasing intersection of simple \Cs -algebras does not need to be simple,
even when each term is isomorphic to $\mathcal{O}_2$.
\begin{Exm}
Let $\mathbb{Z}\curvearrowright \mathbb{T}$ be an irrational rotation.
For each $n\in \mathbb{N}$, set $\Gamma_n:=2^n\mathbb{Z}$.
Then each \Cs -algebra $C(\mathbb{T})\rtimes \Gamma_n$
is an irrational rotation algebra and their intersection is $C(\mathbb{T})$. Now the tensor product trick gives
a decreasing sequence of isomorphs of the Cuntz algebra $\mathcal{O}_2$ whose
intersection is $C(\mathbb{T})$.
\end{Exm}

\appendix
\section{$K$-theory of $C(\partial\mathbb{F}_d/\mathcal{R}_S)\rtimes \mathbb{F}_d$}\label{Sec:App}
In this appendix, we compute the $K$-theory $K_\ast:=(K_0, [1]_0, K_1)$
of the crossed product $C(\partial\mathbb{F}_d/\mathcal{R}_S)\rtimes \mathbb{F}_d$
(cf. Remark \ref{Rem:K-theory}).
Here, as before, $d$ is a natural number greater than $1$
and $S$ denotes the set of all canonical generators of $\mathbb{F}_d$.
For $w\in \mathbb{F}_d$, we denote by $p[w]$
the characteristic function of the clopen set 
$$\left\{(x_n)_{n=1}^\infty \in \partial \mathbb{F}_d: x_1\cdots x_{|w|} =w\right\}$$
and set $q[w]:=p[w]+p[w^{-1}].$
Throughout this appendix, we identify $C(\partial\mathbb{F}_d/\mathcal{R}_S)$
with the $\mathbb{F}_d$-\Cs-subalgebra of $C(\partial\mathbb{F}_d)$ in the canonical way.
Under this identification, it is not difficult to check that for $s\in S$,
$q[s]$ is contained in $C(\partial \mathbb{F}_d/ \mathcal{R}_S)$.
We denote the action $\mathbb{F}_d\curvearrowright C(\partial \mathbb{F}_d)$
by $wf$ for $w\in \mathbb{F}_d$ and $f\in C(\partial\mathbb{F}_d)$.
\begin{Lem}\label{Lem:A1}
The \Cs -algebra $C(\partial\mathbb{F}_d/\mathcal{R}_S)$ is generated by the set 
$$\mathcal{P}:=\{wq[s]: w\in \mathbb{F}_d, s\in S\}.$$
In particular, the space $\partial\mathbb{F}_d/\mathcal{R}_S$
is homeomorphic to the Cantor set.
\end{Lem}
\begin{proof}
By the Stone--Weierstrass theorem, it suffices to show that
the set $\mathcal{P}$ separates the points of $\partial \mathbb{F}_d/\mathcal{R}_S$.
Let $x=(x_n)_{n=1}^\infty$ and $y=(y_n)_{n=1}^\infty$ be two elements
in $\partial\mathbb{F}_d$ satisfying $(x, y)\not\in \mathcal{R}_S$.
If $x\not\in \{ws^{+\infty}: w\in \mathbb{F}_d, s\in S\cup S^{-1}\}$,
then take $n\in \mathbb{N}$ with $x_n \neq y_n$.
Let $m$ be the smallest natural number greater than $n$
satisfying $x_{m}\neq x_n$ (which exists by assumption).
Then the projection $(x_1\cdots x_{m-1})(q[x_m])$ separates $x$ and $y$.
Next consider the case $x=zs^{+\infty}, y=wt^{+\infty}$, where $s, t\in S\cup S^{-1}$ and
$z$, $w$ are elements of $\mathbb{F}_d$ whose last alphabets are not equal to $s^{\pm 1}$, $t^{\pm 1}$, respectively.
Assume $|z|\geq |w|$.
Note that the equality $z=w$ implies $s\neq t^{\pm 1}$ by assumption.
Hence the projection $zq[s]$ separates $x$ and $y$.
Thus $\mathcal{P}$ satisfies the required condition.

The last assertion now follows from the following fact.
A topological space is homeomorphic to the Cantor set
if and only if it is compact, metrizable, totally disconnected, and does not have an isolated point.
\end{proof}
\begin{Lem}\label{Lem:A3}
The $K_0$-group of $C(\partial\mathbb{F}_d/\mathcal{R}_S)\rtimes \mathbb{F}_d$
is generated by $\{[q[s]]_0: s\in S\}$.
\end{Lem}
\begin{proof}
By Lemma \ref{Lem:A1} and the Pimsner--Voiculescu exact sequence \cite{PV},
the $K_0$-group is generated by the elements represented by a projection in $C(\partial\mathbb{F}_d/\mathcal{R}_S)$.
Let $r$ be a projection in $C(\partial\mathbb{F}_d/\mathcal{R}_S)$.
Then there is $n\in \mathbb{N}$ and a set $F$ of elements of $\mathbb{F}_d$ with the length $n$
satisfying
$r=\sum_{w\in F} p[w]$.
For $w\in \mathbb{F}_d$ whose reduced form is $s_{i(1)}^{n(1)}\ldots s_{i(k)}^{n(k)}$
with $s_{i(k-1)}\neq s_{i(k)}^{\pm 1}$, define
$\hat{w} \in \mathbb{F}_d$ by $s_{i(1)}^{n(1)}\ldots s_{i(k-1)}^{n(k-1)}s_{i(k)}^{-n(k)}$.
We will show that $w\in F$ implies $\hat{w}\in F$.
Indeed, if $w\in F$, then $r(ws_{i(k)}^{+\infty})=1$.
Hence we must have $r(ws_{i(k)}^{-\infty})=1$. This implies $\hat{w}\in F$ as desired.
Since $w\neq \hat{w}$ and $[p[w]+p[\hat{w}]]_0=[q[s_{i(k)}^{n(k)}]]_0$,
it suffices to show the claim for elements of the form $[q[s^n]]_0$ with $s\in S, n\in \mathbb{N}$.
This follows from the equations
$q[s^2]=sq[s]+s^{-1}q[s]+q[s]-2$
and
$q[s^k]=sq[s^{k-1}]+s^{-1}q[s^{k-1}]-q[s^{k-2}]$
for $s\in S$ and $k>2$.
\end{proof}

\begin{Thm}\label{Thm:K}
The
$K_\ast(C(\partial\mathbb{F}_d/\mathcal{R}_S)\rtimes \mathbb{F}_d))$
is isomorphic to $(\mathbb{Z}^d, (1, 1, \ldots, 1), \mathbb{Z}^d)$.
\end{Thm}
\begin{proof}
We first show the isomorphism of the pair $(K_0, [1]_0)$.
By Lemma \ref{Lem:A3}, it suffices to show the linear independence of
the family $([q[s]]_0)_{s \in S}$.
Let $\eta\colon C(\partial\mathbb{F}_d, \mathbb{Z})^{\oplus S} \rightarrow C(\partial\mathbb{F}_d)$
be the additive map defined by
$(f_s)_{s\in S}\mapsto \sum_{s\in S} (f_s-s(f_s))$ and denote by $\tau$ the restriction of $\eta$
to $C(\partial\mathbb{F}_d/\mathcal{R}_S, \mathbb{Z})^{\oplus S}$.
Then the Pimsner--Voiculescu exact sequence \cite{PV} shows that
the canonical map
$C(\partial\mathbb{F}_d/\mathcal{R}_S, \mathbb{Z}) \rightarrow K_0(C(\partial\mathbb{F}_d/\mathcal{R}_S)\rtimes \mathbb{F}_d)$
is surjective and its kernel is equal to ${\rm im}(\tau)$.
Hence it suffices to show that ${\rm im}(\tau)$ does not contain
a nontrivial linear combination of the projections $q[s], s\in S$.
The isomorphisms $\ker(\eta)\cong K_1(C(\partial\mathbb{F}_d)\rtimes \mathbb{F}_d)\cong \mathbb{Z}^d$ (see \cite{Cun, PV, Spi})
show that $\ker(\eta)=\{(f_s)_{s\in S}: {\rm each\ }f_s{\rm\ is\ constant}\}$.
Now let $r=\sum_{s\in S}n(s)q[s]$ be a nontrivial linear combination of $q[s]$'s.
If $\sum_{s\in S} n(s) \not\equiv 0 \mod (d-1)$, then
$r \not \in {\rm im}(\eta)$ by \cite{Cun, Spi}.
If $\sum_{s\in S} n(s) = (d-1)m$ for some $m\in \mathbb{Z}$,
then $\sum_{s\in S}n(s)q[s]=\eta((g_{s})_{s \in S})$,
where $g_s:=(n(s)-m)p[s^{-1}]$ for $s\in S$.
Hence $\eta^{-1}(\{r\})=(g_s)_{s\in S}+ \ker(\eta)$,
which does not intersect with $C(\partial\mathbb{F}_d/\mathcal{R}_S, \mathbb{Z})^{\oplus S}$.
Thus we have $r\not\in {\rm im}(\tau)$ in either case.

The isomorphism of the $K_1$-group follows from the Pimsner--Voiculescu exact sequence \cite{PV}
and the equality $\ker(\tau)=\ker(\eta)$.
\end{proof}
\begin{Rem}\label{Rem:RW}
Let $\mathfrak{F}\subset S$ be a nonempty proper subset.
Then a similar proof to the above shows that $\partial \mathbb{F}_d/\mathcal{R}_{\mathfrak{F}}$ is
the Cantor set and $K_\ast(C(\partial\mathbb{F}_d/\mathcal{R}_{\mathfrak{F}})\rtimes\mathbb{F}_d)$
is isomorphic to $(\mathbb{Z}^d, (1, 0, \ldots, 0), \mathbb{Z}^d)$.
(The set
$\{ [1]_0, [p[s]]_0, [q[t]]_0: s\in S\setminus \mathfrak{F}, t\in \mathfrak{F}'\}$
is a basis of the $K_0$-group for any subset $\mathfrak{F}'$ of $\mathfrak{F}$ with the cardinality $\sharp\mathfrak{F}-1$.
We remark that the equality $(d-\sharp \mathfrak{F}-1)[1]_0=-\sum_{s\in \mathfrak{F}}[q[s]]_0$ holds
in the $K_0$-group.)
\end{Rem}
The classification theorem of Kirchberg and Phillips \cite{Kir, Phi},
Theorem \ref{Thm:K}, Remarks \ref{Rem:Kirchberg} and \ref{Rem:RW} show that
the crossed products of the dynamical systems $\varphi_{\mathfrak{F}}\colon \mathbb{F}_d\curvearrowright \partial\mathbb{F}_d/\mathcal{R}_{\mathfrak{F}}$ are mutually isomorphic for nonempty subsets $\mathfrak{F}$ of $S$.
We next show that they are not continuously orbit equivalent (or equivalently,
their transformation groupoids are not isomorphic as $\acute{{\rm e}}$tale groupoids; see \cite[Definition 5.4]{Suz} for the definition and \cite{Suz} for a relevant topic).
Hence we obtain examples of amenable minimal Cantor $\mathbb{F}_d$-systems
which are not continuously orbit equivalence but have isomorphic crossed products.
In the proof, we use the notion of the topological full group of Cantor systems.
For the definition, see \cite[Definition 5.5]{Suz} for example.
\begin{Thm}\label{Thm:CE}
Let $\mathfrak{F}$ and $\mathfrak{F}'$ be nonempty subsets of $S$.
Then $\varphi_{\mathfrak{F}}$ and $\varphi_{\mathfrak{F}'}$ are continuously orbit equivalence
if and only if $\sharp \mathfrak{F}=\sharp \mathfrak{F}'$.
\end{Thm}
\begin{proof}
Let $\mathfrak{F}$ be a nonempty subset of $S$
and denote by $[y]$ the element of $\partial \mathbb{F}_d /\mathcal{R}_{\mathfrak{F}}$ represented by $y\in \partial \mathbb{F}_d$.
Set
$$X:=\left\{x\in \partial\mathbb{F}_d/\mathcal{R}_\mathfrak{F}: {\rm there\ is\ }g\in \mathbb{F}_d\setminus \{e\}{\rm\ with\ }g.x=x \right\}.$$
Notice that the definition of $X$ only depends on the structure of the
transformation groupoid
$\partial\mathbb{F}_d/\mathcal{R}_{\mathfrak{F}}\rtimes \partial\mathbb{F}_d$.
It is clear from the definition that
$X=\{[w^{+\infty}]: w\in \mathbb{F}_d\setminus\{e\}\}$. 
For each element $x\in X$, consider the following condition.
\begin{itemize}
\item[$(\ast)$]There is $F \in [[\varphi_{\mathfrak{F}}]]$
which fixes $x$ and acts nontrivially on any neighborhood of $x$,
and both $(F^{n})_{n=1}^\infty$ and  $(F^{-n})_{n=1}^\infty$
do not uniformly converge to the constant map $y\mapsto x$ on any neighborhood of $x$.
\end{itemize}
Then it is easy to check that
$$Y:=\left\{x\in X: x{\rm\ satisfies\ the\ condition\ }(\ast)\right\}=\left\{[gw^{+\infty}]: g\in \mathbb{F}_d, w\in \mathfrak{F}\right\}.$$
Now the cardinality of $\mathfrak{F}$
is recovered as the number of the $\mathbb{F}_d$-orbits in $Y$.
\end{proof}
\begin{Rem}
It follows from Theorem \ref{Thm:CE} and Matui's theorem \cite[Theorem 3.10]{Mat} that
the topological full groups of $\varphi_{\mathfrak{F}}$ and $\varphi_{\mathfrak{F'}}$
are not isomorphic when $\sharp \mathfrak{F} \neq \sharp\mathfrak{F}'$.
\end{Rem}
\begin{Rem}
The induced construction of dynamical systems (see \cite[Theorem 4.7]{Suz} for instance) and the proof of Theorem \ref{Thm:CE} show that any finitely generated virtually free group $\Gamma$
admits amenable minimal topologically free Cantor $\Gamma$-systems
which are not continuously orbit equivalence but have isomorphic crossed products.
\end{Rem}
\subsection*{Acknowledgement}
The author would like to thank to Hiroki Matui,
who turns the author's interest to amenable quotients of the boundary actions.
He also thanks to Narutaka Ozawa
for helpful discussions on hyperbolic groups and approximation theory.
He is supported by Research Fellow
of the JSPS (No.25-7810) and the Program of Leading Graduate Schools, MEXT, Japan.

\end{document}